\documentclass[12pt,reqno]{amsart}

\usepackage{amsmath,amsfonts,amssymb,amscd,amsthm,calc,enumerate}
\usepackage[english]{babel}
\usepackage{yfonts,color}

\newtheorem{theorem}{Theorem}[section]
\newtheorem{proposition}[theorem]{Proposition}
\newtheorem{lemma}[theorem]{Lemma}
\newtheorem{corollary}[theorem]{Corollary}

\newtheorem{question}[theorem]{Question}

\theoremstyle{definition}
\newtheorem{definition}[theorem]{Definition}

\newcommand{\darrow}{\!\downarrow}
\newcommand{\uarrow}{\!\uparrow}
\newcommand{\la}{\langle}
\newcommand{\ra}{\rangle}

\newcommand{\cmp}[1]{\overline #1}

\newcommand{\restr}{\mbox{\raisebox{.5mm}{$\upharpoonright$}}}

\renewcommand{\leq}{\leqslant}

\newcommand{\fa}{\forall}
\newcommand{\ex}{\exists}
\newcommand{\vph}{\varphi}

\newcommand{\A}{\mathcal{A}}
\newcommand{\K}{\mathcal{K}}
\newcommand{\U}{{\sf{U}}}
\newcommand{\true}{{\sf{T}}}
\newcommand{\false}{{\sf{F}}}
\newcommand{\ifthenelse}{{\text{\sf if-then-else}}}
\newcommand{\zero}{{\text{\sf zero}}}

\newcommand{\dom}{\mathrm{dom}}

\begin{document}

\title[Computability in pca's]
{Computability in partial combinatory algebras}

\author[S. A. Terwijn]{Sebastiaan A. Terwijn}
\address[Sebastiaan A. Terwijn]{Radboud University Nijmegen\\
Department of Mathematics\\
P.O. Box 9010, 6500 GL Nijmegen, the Netherlands.
} \email{terwijn@math.ru.nl}

\begin{abstract} 
We prove a number of elementary facts about computability
in partial combinatory algebras (pca's). 
We disprove a suggestion made by Kreisel about using  
Friedberg numberings to construct extensional pca's.
We then discuss separability and elements without total extensions.
We relate this to Ershov's notion of precompleteness, and we show 
that precomplete numberings are not 1-1 in general.
\end{abstract}



\date{\today}

\maketitle

\section{Introduction}

Combinatory algebra was founded by Sch\"{o}nfinkel~\cite{Schoenfinkel}
and Curry~\cite{Curry}, and is closely related to the lambda calculus
(cf.\ Barendregt~\cite{Barendregt}).
Curry attempted to use combinatory algebra as a foundation of mathematics, 
and Church tried the same for the lambda calculus. 
Both attempts fell short (Church's system was inconsistent 
and Curry's was too weak), but the formalisms became important for 
other reasons, for example as foundational theories for the theory of 
computation.  
Partial combinatory algebra (pca) was first studied in 
Feferman~\cite{Feferman} as an axiomatic approach to the theory 
of computation and the study of various constructive theories. 
See Troelstra and van Dalen~\cite{TroelstravanDalenII} for a discussion 
of pca and the relation with constructive mathematics, as well as 
a varied list of models of pca. 
In this paper we discuss computability in pca's and relate this to 
the theory of numberings. 

The work in this paper is related to several other approaches to abstract models
of computation, many of which are discussed in 
Longley and Normann~\cite{LongleyNormann}.
In particular there is the notion of a {\em Basic Recursive Function Theory\/} (BRFT), 
introduced by Wagner and Strong, which is closely related to 
Moschovakis' notion of a precomputation theory (cf.\ Odifreddi~\cite[p222]{Odifreddi}).
Every BRFT gives rise to a pca, as pointed out in \cite[p199]{CockettHofstra}. 
This will be relevant below when we discuss the work of Kreisel. 

Kreisel~\cite{Kreisel1971} eloquently discusses some of the reasons one might 
want to generalize the setting of classical computability theory.
Kreisel's ideas were highly  influential, cf.\ Sacks~\cite{Sacks},
and also the review by Yates~\cite{Yates}.
Hyland wrote \cite{Hyland} as a kind of tribute to \cite{Kreisel1971}.
In section~\ref{sec:Kreisel} we disprove a suggestion made by 
Kreisel in \cite{Kreisel1971} 
(quoted as a theorem in Odifreddi~\cite{Odifreddi}) 
about constructing extensional models, 
using Friedberg's result that the partial computable functions are 
computably enumerable without repetitions. 
We prove that such a construction is impossible.

Cockett and Hofstra \cite{CockettHofstra} discuss category theoretic
approaches to computability theory and pca's. 
They introduce the notion of a Turing category, relaxing restrictions 
in earlier work on so-called recursion categories by 
Di Paola and Heller, and then proceed to show that the study of 
Turing categories is essentially (in a precise sense) equivalent to 
the study of pca's. 

The paper is organized as follows. 
In section~\ref{sec:pca} we list some preliminaries about pca's,
and in section~\ref{sec:comp} we discuss basics of computable and 
computably enumerable (c.e.) sets in pca's. 
In section~\ref{sec:Post} we prove that Post's theorem fails in 
Kleene's second model $\K_2$. 
In section~\ref{sec:relativization} we discuss the halting problem 
and relativization.
In section~\ref{sec:Kreisel} we discuss Kreisel's suggestion about 
Friedberg numberings and extensional pca's, and show that it is impossible. 
In section~\ref{sec:insep}  we discuss inseparable sets, and 
in section~\ref{sec:tot} use this in the discussion of elements 
without total extensions. 
In section~\ref{sec:numberings} we relate this to the theory of numberings, 
and show that precomplete generalized numberings cannot be~1-1. 
In section~\ref{sec:Arslanov} we use the notion of relativization 
to formulate an analog of Arslanov's completeness criterion for pca's, 
motivated by results about the theory of numberings. 

Our notation from computability theory is mostly standard
and follows Odifreddi~\cite{Odifreddi} and Soare~\cite{Soare}.
The natural numbers are denoted by $\omega$. 
$\omega^{<\omega}$ is the set of finite sequences over $\omega$.
$\vph_e$, $e\in\omega$, denotes a standard enumeration of the 
(unary) partial computable (p.c.) functions. 
Notation for pca's is introduced in section~\ref{sec:pca}.
For a partial function $\vph$, $\dom(\vph)$ denotes the set of 
elements where $\vph$ is defined, and 
$\vph\restr x$ denotes its restriction to numbers $n<x$.

\section{Partial combinatory algebras} \label{sec:pca}

\begin{definition} \label{def:pca}
A {\em partial applicative structure\/} (pas) is a set $\A$ together
with a partial map $\cdot$ from $\A\times \A$ to $\A$.  
We also write $ab$ instead of $a\cdot b$, and think of this as 
`$a$ applied to $b$'. If this is defined we denote this by $ab\darrow$.  
By convention, application associates to the left. We write $abc$ instead
of $(ab)c$.  {\em Terms\/} over $\A$ are built from elements of $\A$,
variables, and application. If $t_1$ and $t_2$ are terms then so is
$t_1t_2$.  If $t(x_1,\ldots,x_n)$ is a term with variables $x_i$, and
$a_1,\ldots,a_n \in\A$, then $t(a_1,\ldots,a_n)$ is the term obtained
by substituting the $a_i$ for the~$x_i$.  For closed terms
(i.e.\ terms without variables) $t$ and $s$, we write $t \simeq s$ if
either both are undefined, or both are defined and equal. 
Here application is \emph{strict} in the sense that for $t_1t_2$ to be 
defined, it is required that both $t_1,t_2$ are defined.
We say that an element $f\in \A$ is {\em total\/} if $fa\darrow$ for
every $a\in \A$.

A pas $\A$ is {\em combinatory complete\/} if for any term
$t(x_1,\ldots,x_n,x)$, $0\leq n$, with free variables among
$x_1,\ldots,x_n,x$, there exists a $b\in \A$ such that
for all $a_1,\ldots,a_n,a\in \A$,
\begin{enumerate}[\rm (i)]

\item $ba_1\cdots a_n\darrow$,

\item $ba_1\cdots a_n a \simeq t(a_1,\ldots,a_n,a)$.

\end{enumerate}
A pas $\A$ is a {\em partial combinatory algebra\/} (pca) if
it is combinatory complete.
\end{definition}

Note that combinatory completeness is the analog of the 
S-m-n-theorem (also called the parametrization theorem)
from computability theory, cf.\ Odifreddi~\cite{Odifreddi}.

\begin{theorem} {\rm (Feferman~\cite{Feferman})} \label{Feferman}
A pas $\A$ is a pca if and only if it has elements $k$ and $s$
with the following properties for all $a,b,c\in\A$:
\begin{itemize}

\item $k$ is total and $kab = a$, 

\item $sab\darrow$ and $sabc \simeq ac(bc)$.

\end{itemize}
\end{theorem}

Note that $k$ and $s$ are nothing but partial versions of the 
familiar combinators from combinatory algebra.
As noted in \cite[p95]{Feferman}, Theorem~\ref{Feferman} has the 
consequence that in any pca we can define lambda-terms in the usual 
way (cf.\ Barendregt~\cite[p152]{Barendregt}):\footnote{%
Because the lambda-terms in combinatory algebra do not have
the same substitution properties as in the lambda calculus, 
we use the notation $\lambda^*$ rather than~$\lambda$,
cf.\ Barendregt~\cite[p152]{Barendregt}.
E.g.\ (\cite[p84]{LongleyNormann}) the terms 
$\lambda x.(\lambda y.y)x)$ and $\lambda x.x$ are $\beta$-equivalent, 
but their $\lambda^*$-versions are $s(ki)i$ and $i$, and these 
are in general different elements in a pca.} 
For every term $t(x_1,\ldots,x_n,x)$, $0\leq n$, with free variables among
$x_1,\ldots,x_n,x$, there exists a term $\lambda^* x.t$ 
with variables among $x_1,\ldots,x_n$,
with the property that for all $a_1,\ldots,a_n,a\in\A$,
\begin{itemize}

\item $(\lambda^* x.t)(a_1,\ldots, a_n)\darrow$,

\item $(\lambda^* x.t)(a_1,\ldots, a_n)a \simeq t(a_1,\ldots,a_n,a)$.

\end{itemize}

The most famous examples of a pca are Kleene's first and second 
models $\K_1$ and $\K_2$. 
$\K_1$ consists of the natural numbers $\omega$, with application
defined as $n\cdot m = \vph_n(m)$. So this is essentially the 
setting of classical computability theory. 
$\K_2$ is defined on $\omega^\omega$, with application $\alpha\cdot\beta$ 
defined by applying the continuous functional with code $\alpha$ to the 
real $\beta$. See Longley and Normann~\cite{LongleyNormann} for more 
details. Many other examples of pca's can be found in the books by 
Beeson~\cite{Beeson}, Odifreddi~\cite{Odifreddi}, and van Oosten~\cite{vanOosten}.

The presence of the $\lambda^*$-terms and the combinators allows for the 
following definitions in any pca (cf.\ Barendregt~\cite[p44]{Barendregt} 
and van Oosten \cite{vanOosten}):
The Booleans true and false can be defined as 
$\true = \lambda^* xy.x = k$ and $\false = \lambda^* xy.y = ki$, 
where $i=skk$.
We can implement definition by cases using an element $\ifthenelse$ 
with the property
$\ifthenelse \true ab = a$ and 
$\ifthenelse \false ab = b$.
Namely, we can simply take $\ifthenelse = i$. 
This also gives the Boolean operations,  
for example 
\begin{align*}
{\sf not} a &= \ifthenelse a\false\true, \text{and} \\
{\sf and} ab &= \ifthenelse a (\ifthenelse b\true\false)\false.
\end{align*}

Coding of sequences is a standard device in the lambda calculus. 
Using the $\lambda^*$-terms available in any pca, we can 
code $n$-tuples $(a_1,\ldots,a_n)$ by 
$\langle a_1,\ldots, a_n\rangle=\lambda^* z.za_1\ldots a_n$.
The inverse projection functions can be defined as 
$\U^n_i=\lambda^* u_1\ldots u_n.u_i$, so that 
$$
\langle a_1,\ldots, a_n\rangle \U^n_i=a_i.
$$

There are various ways to define the natural numbers 
$\bar 0, \bar 1, \bar 2,\ldots$ in a pca. 
A convenient way is to define   
$\bar 0 = i$, and $\overline{n+1} = \la \false,\bar n\ra$, 
cf.\ Barendregt~\cite[p44]{Barendregt}.

All the above can be defined in any pca, but they may trivialize
if $|\A| = 1$. 
van Oosten~\cite[p11]{vanOosten} calls $\A$ {\em nontrivial\/} if $|\A|>1$. 
We note that $n=1$ is the only possible cardinality for a finite pca:

\begin{proposition}
Suppose that a pca $\A$ is finite. Then $|\A|=1$.
\end{proposition}
\begin{proof}
Note that every pca is nonempty, since by Feferman's Theorem~\ref{Feferman} 
it has to contain the combinators $k$ and~$s$. 
Furthermore, there exists a (total) pca with precisely one element $a$, 
with application $aa\darrow = a$. 
In this pca we have $s = k = a$. 
Since all $\lambda^*$-terms are equal to $a$, also $\true = \false = a$.

Now suppose that $\A = \{a_1,\ldots, a_n\}$, and $n>1$.
$\A$ contains the elements $ka_1, \ldots, ka_n$, which are $n$ 
distinct constant functions since $ka_i b = a_i$. 
$\A$ also contains the identity function $i = skk$, which is not 
a constant function since $n>1$. So $\A$ has at least $n+1$ elements, 
a contradiction.
\end{proof}

Following \cite{vanOosten2006}, we say that a partial function 
$\vph:\A\rightarrow \A$ is {\em representable\/} in $\A$ if there 
is an element $r\in\A$ such that for every $a\in\dom(\vph)$, 
$ra\darrow = \vph(a)$. 
We have a similar definition for multivariate functions.

\section{Computable sets and c.e. sets in pca's}\label{sec:comp}

The following definition is taken from 
van Oosten and Voorneveld \cite{vanOostenVoorneveld}, 
which in turn is based on Longley~\cite{Longley}.

\begin{definition} \label{def:decidable}
Let $\A$ be a pca, and $A\subseteq\A$. 
$A$ is called {\em decidable\/} in $\A$ if there exists a total 
$c\in \A$ such that for every $a\in\A$, 
\begin{align*}
ca = \true  &\Longleftrightarrow a\in A, \\
ca = \false &\Longleftrightarrow a\notin A.
\end{align*}
Note that this is equivalent to saying that the characteristic 
function $\chi_A : \A \rightarrow \{\true,\false\}$ 
of $A$ is representable in $\A$. 
Instead of $\true$ and $\false$, we may equivalently use 
$0$ and $1$, cf.\ Proposition~\ref{prop:TF01}.
\end{definition}

We can also easily define the analog of c.e.\ sets in the following way.

\begin{definition}
We say that $A$ is {\em computably enumerable (c.e.)\/} in $\A$ if 
there exists $e\in\A$ such that 
$$
A = \dom(e) = \{a\in\A \mid ea\darrow\}.
$$
\end{definition}

Note that this notion is not very useful in total pca's, 
since there $\A$ itself is the only c.e.\ set.\footnote{
In total pca's, such as the lambda calculus, one can represent 
`undefined' in other way's, for example using terms without 
normal form, cf.\ Barendregt~\cite{Barendregt1992}.} 
For nontotal pca's we have the following result.

\begin{proposition} \label{decimpce}
In nontotal pca's, decidable sets are c.e.
The converse implication does not hold in general. 
\end{proposition}
\begin{proof}
Suppose that $\A$ is a nontotal pca, and that $A\subseteq\A$ and 
$c\in\A$ are as in Definition~\ref{def:decidable} above.
First note that $\A$ contains a totally undefined function. 
Namely, since $\A$ is nontotal, there are $f,g\in\A$ such that $fg\uarrow$.
Now define $h=\lambda^* x.fg = s(kf)(kg)$. 
Then $ha\uarrow$ for every~$a\in\A$.

Now define $ea = \ifthenelse (ca) 0 (ha)$. 
This yields $0$ if $ca = \true$, so if $a\in A$, 
and $ha$, which is undefined, otherwise. 
Hence $ea\darrow$ if and only if $a\in A$. 

For the converse implication, 
c.e.\ sets are not always decidable by 
Proposition~\ref{HP}. 
\end{proof}

Definition~\ref{def:decidable} uses $\true,\false$ as truth values.
In mathematics it is customary to use $0,1$ as values of 
characteristic functions. We show that in the context of pca's, 
we may equivalently use the numerals $\bar 0,\bar 1$ 
as truth values.

\begin{proposition} \label{prop:TF01}
Let $\A$ be a pca. 
There exists $c,d\in\A$ such that  
\begin{align*}
c\true  &= \bar 1 & d\bar 1 &= \true \\
c\false &= \bar 0 & d\bar 0 &= \false.
\end{align*}
\end{proposition}
\begin{proof}
It is easy to check that 
$c = \lambda^* z. \ifthenelse z \bar 1 \bar 0$
satisfies the first part of the proposition. 

For the second part, note that the term 
$\zero = \lambda^* x. x\true$ has the property 
$\zero \bar 0 = \true$ and $\zero \bar 1 = \false$ 
(cf.\ \cite[p134]{Barendregt})
so we can take $d$ to be the term $\lambda^* x. {\sf not}(\zero x)$.
\end{proof}

Since in every pca there are elements mapping $\true$ and $\false$ to 
$\bar 1$ and $\bar 0$, and vice versa, we may equivalently use 
$\bar 1$ and $\bar 0$ in Definition~\ref{def:decidable}. 
From now on we will mostly use the latter, and simply write $0$ and~$1$ 
for the values of characteristic functions.

\section{A counterexample to Post's theorem}\label{sec:Post}

Post's theorem is the statement that for $A\subseteq\omega$, 
if both $A$ and its complement $\cmp{A}$ are c.e., then 
$A$ is decidable. To decide whether $x\in A$, simply enumerate 
both $A$ and $\cmp{A}$ until $x$ appears in one of them. 
This works because in $\omega$, c.e.\ sets have finite 
approximations, and if $x\in A$ then $x$ appears in $A$ 
after finitely many steps. 
In general, we do not have a good notion of approximation in pca's, 
and being ``enumerated'' into a c.e.\ set does not have to happen 
in finitely many stages in every pca. Hence there does not seem to 
be a reason why Post's theorem should hold in general. 
Indeed we now show that it fails in Kleene's second model~$\K_2$.
 
\begin{proposition} \label{Postfails}
Post's theorem fails in~$\K_2$.
\end{proposition}
\begin{proof} 
Let $\bar{0}$ denote the all zero sequence in $2^\omega$, 
and let $\bar{1}$ denote the all one sequence. 

Let $A = \{\bar{0}\}$. 
Then $A$ is c.e.\ in $\K_2$: 
Define $\hat\alpha: 2^{<\omega} \rightarrow 2^\omega$ by
$$
\hat\alpha(x) = 
\begin{cases}
0^n &\text{if $x = 0^n$}\\
\uparrow &\text{otherwise}.
\end{cases}
$$
$\hat\alpha$ defines a partial computable functional 
$\alpha:2^\omega\rightarrow 2^\omega$ that simply copies 
the input, as long as the input consists of only zeros, 
and becomes undefined otherwise. 
Hence $\dom(\alpha) = \bar{0}$, which shows that $A$ is c.e.

The complement $\cmp{A} = \{r\in 2^\omega : r\neq\bar{0}\}$ 
is also c.e.:
Define 
$$
\hat\beta(x) = 
\begin{cases}
1^n &\text{if $n=|x|$ and $x(i)=1$ for some $i<n$,}\\
\uparrow &\text{otherwise}.
\end{cases}
$$
Then $\hat\beta$ defines a partial computable functional 
$\beta$ with $\dom(\beta) = \cmp{A}$.

So both $A$ and $\cmp{A}$ are c.e., but 
$A$ is not decidable in $\K_2$. Suppose that it were, 
and suppose that $\gamma$ is a computable functional 
such that 
$\gamma(x) = \true$ if $x = \bar{0}$, and 
$\gamma(x) = \false$ if $x \neq \bar{0}$.
Now since $\gamma$ is continuous, this distinction has to 
be made on the basis of a finite initial segment of $x$, 
which is impossible. Note that for this argument it does 
not really matter what $\true$ and $\false$ are, as long 
as they are distinct reals in $2^\omega$. 
\end{proof}

\section{Reductions and relativization} \label{sec:relativization}

In Beeson~\cite[p107]{Beeson} it is already remarked that besides 
the existence of a universal function and the undecidability of the 
halting problem, not many analogues of classical results in 
computability theory can be proved. 
We claim no originality for the results in this section, but 
for the record discuss the m-completeness of the halting problem. 
This was surely known to people working in axiomatic recursion theory, 
but since we have not been able to locate it in the literature, 
we include it here. In any case, it is an easy fact that is 
completely analogous to Turing's classical result.

We can define the analog of the halting problem in any pca~$\A$ using the 
coding of sequences. 
Define 
$$
H = \{\la a,b\ra \mid ab\darrow\}.
$$ 

\begin{proposition} \label{HP}
For every nontotal pca $\A$, $H$ is undecidable and c.e.\ in~$\A$. 
\end{proposition}
\begin{proof}
To see that $H$ is c.e.\ in $\A$, define $e\in\A$ by 
$$
e x = (x\U^2_1)(x\U^2_2), 
$$
where $\U^2_i$ refers to the projection functions defined in 
section~\ref{sec:pca}. We then have in particular that 
$$
e \la a,b\ra \darrow \;\Longleftrightarrow\; ab\darrow 
\;\Longleftrightarrow\; \la a,b\ra \in H 
$$
which shows that $H$ is indeed c.e. in $\A$.

The proof of the undecidability is the same as for the classical case.
Namely suppose that $H$ were decidable. 
This would mean the existence of $f\in\A$ such that
$$
f\la a,b\ra = 
\begin{cases}
\true  &\text{if $ab\darrow$,}\\
\false &\text{if $ab\uarrow$.}
\end{cases}
$$
Define $g\in \A$ such that $ga\darrow$ if and only if $f\la a,a\ra = \false$. 
Such $g$ can be defined using the $\ifthenelse$ operator 
(cf.\ section~\ref{sec:pca}) as follows. 
Let $ha\uarrow$ for every $a$. (Such $h$ exists in any nontotal pca, 
cf.\ the proof of Proposition~\ref{decimpce}. Now define
$$
ga = \ifthenelse ({\sf not}(f\la a,a\ra))0(ha)
$$
This yields $0$ if $f\la a,a\ra = \false$, hence if $aa\uarrow$, 
and $ha$, which is undefined, otherwise.
Hence $ga\darrow$ if and only if $aa\uarrow$. 
Taking $a=g$ we obtain a contradiction.
\end{proof}

Note that by Proposition~\ref{HP}, as soon as a pca has 
{\em one\/} undefined application $ab\uarrow$, its 
halting problem is undecidable. 

Defining the analog of m-reductions is also straightforward:

\begin{definition}
For sets $A,B\subseteq \A$, we say that $A$ {\em m-reduces\/} to $B$, 
denoted $A\leq_m B$, if there exists a total element $f\in\A$
such that 
$$
a\in A \Longleftrightarrow fa \in B
$$
for every $a\in \A$. 
We write $A\equiv_m B$ if both $A\leq_m B$ and $B\leq_m A$, in 
which case we say that $A$ and $B$ have the same {\em m-degree}.
\end{definition}

Many basic properties of m-reductions in $\omega$ carry over to the
general case. 
For example, we can define the {\em diagonal halting problem\/} 
$$
K = \{a \mid aa\darrow\},
$$ 
and show that $K\equiv_m H$. 

\begin{proposition}
$H$ is m-complete for the c.e.\ sets in $\A$, i.e.\ 
$A\leq_m H$ for every such set $A$.
\end{proposition}
\begin{proof}
Suppose that $A$ is c.e.\ in $\A$, say $A = \dom(e)$ for $e\in \A$.
Then 
$$
e\in A \Leftrightarrow ea\darrow \Leftrightarrow \la e,a\ra\in H, 
$$
and hence $fa = \la e,a\ra$ is an m-reduction from $A$ to $H$. 
Note that $f$ is total, since $\la \cdot\, ,\cdot\ra$ is implemented 
by $\lambda^*$-terms in any pca. 
\end{proof}

We also have an analog of Turing reductions in any pca $\A$. 
This is somewhat harder to define, and was carried out in 
van Oosten \cite{vanOosten2006}, see also \cite{vanOostenVoorneveld}.
This gives for any pca $\A$ and any partial function $f:\A\rightarrow \A$
a new pca $\A[f]$ in which $f$ is represented, in such a way that
$\A[f]$ contains $\A$ in a natural way. Application in $\A[f]$ models 
computation in $\A$ with $f$ as an oracle, and thus 
provides an analog of relativization for pca's. 
For $\A=\K_1$ we have that 
$A \in \K_1[B]$ is equivalent to Turing reducibility $A\leq_T B$.

More specifically, $\A[f]$ has the same underlying set as $\A$, with only a different 
application operator $\cdot_f$, defined as follows. 
$a \cdot_f b \darrow = c$ if there exist $e_0,\ldots, e_{n-1} \in \A$ 
(the queries to the oracle $f$) such that for every $i<n$:
\begin{itemize}

\item $a \cdot \la b,f(e_0),\ldots,f(e_{i-1}\ra = \la \false,e_i\ra$,

\item $a  \cdot \la b,f(e_0),\ldots,f(e_{n-1}\ra = \la \true,c\ra$.

\end{itemize}
Here $\cdot$ denotes application in $\A$.
We will use this construction  in 
sections~\ref{sec:tot} and \ref{sec:Arslanov}.

\section{Extensionality and enumerations without repetitions}
\label{sec:Kreisel}
 
\begin{definition}
A pca $\A$ is called {\em extensional\/} if
\begin{equation*} \label{ext}
\fa a\in\A (fa \simeq ga) \Longrightarrow f=g
\end{equation*}
for all $f,g\in \A$.
\end{definition}

In \cite{Friedberg} Friedberg proved the classic result that the 
class of partial computable functions is c.e.\ without repetitions.
In the second edition of Odifreddi~\cite[p224]{Odifreddi} it is stated that
the existence of an extensional partial combinatory algebra follows 
from Friedberg's result, 
with a reference to Kreisel \cite{Kreisel1971}.
(This is from the second edition of \cite{Odifreddi}; 
in the first edition this statement is missing.) 
Indeed, Kreisel (p186 ibid.) suggested that an enumeration without 
repetitions could be used to obtain an extensional model.
(Note however that Kreisel's concern were models of BRFT mentioned above, 
not pca's.) 
However, Kreisel explicitly says that he did not verify this result.\footnote{%
Kreisel discusses the relation between BRFT and set theory, using 
generalized recursion theory. In this context he discusses extensionality.
Kreisel writes: ``...there are two ways of treating {\em extensionality\/}. 
[The first is proof-theoretic] Another is to appeal to an 
{\em enumeration without repetition\/}; but I have not stopped to verify
the obvious essential point whether the axioms of BRFT are in fact satisfied
for such an enumeration without repetition...''}
In any case, it seems that the mere statement of Friedberg's theorem 
is not sufficient to obtain an extensional~pca, so that at least an 
adaptation of the proof of Friedberg's result is required.

Suppose that $\psi_e$, $e\in\omega$, is an enumeration of all unary 
p.c.\ functions. 
On the face of it, it seems plausible that one could make this 
into an extensional pca, since after all every function in the 
enumeration has a unique code. 
Of course the intended application operator here is 
\begin{equation} \label{application}
n\cdot m = \psi_n(m).
\end{equation}
To prove that $\omega$ with this application operator is a pca, 
one has to show that there exist combinators $k$ and $s$ as in 
Theorem~\ref{Feferman}. 
Now the statement of Friedberg's theorem itself is not sufficient 
to prove this. Namely, for every $a$ there is a code $ka$ 
of the constant $a$ function in the enumeration, but we cannot find 
such codes uniformly in~$a$. Hence we cannot prove that we have a 
combinator $k$ satisfying $kab=a$ for every $a$ and~$b$, 
which is the first requirement of Theorem~\ref{Feferman}. 
This obstacle, however, can be overcome by an adaptation of 
the proof of Friedberg's result, see Theorem~\ref{impossible}~(i).
However, for the combinator $s$ no such adaptation is possible.

\begin{lemma} \label{lemma}
There exists a computable enumeration without repetitions $\psi_x$ 
of the unary p.c.\ functions such that for all $x,y\in\omega$, 
\begin{equation}
\psi_{2x+1}(y) = x. \label{k} 
\end{equation}
\end{lemma}
\begin{proof}
For the application of the lemma below, we need to be able to effectively 
retrieve the combinator $k$ (as in Theorem~\ref{Feferman}) 
from the enumeration. Note that for every $a$, $ka$ is the function 
that is constant~$a$. We code the functions $ka$ on the odd numbers by 
defining $\psi_{2a+1}$ as in~\eqref{k}.
We use the even numbers for the construction of the enumeration of 
all other p.c.\ functions, in the manner of Friedberg~\cite{Friedberg}
(see also Odifreddi \cite[II.5.22]{Odifreddi}).\footnote{
There is nothing very special about the class of constant functions 
being fixed in this lemma. Similar modifications of Friedberg's result 
have been made by Pour-El and Howard and others, 
cf.\ \cite[p232]{Odifreddi} for references and further discussion.}

We start the construction by fixing $\psi_{2x+1}$ as in \eqref{k} for 
every~$x$.
We construct $\psi_{2x}$ in such a way that every unary p.c.\ function 
occurs exactly once. 
(We assume that $\vph_e$ is an enumeration of all unary p.c.\ functions.)
We do this by letting $\psi_{2x}$ follow some $\vph_e$ for every~$x$. 
If subsequently it looks like $e$ is not a minimal code of $\vph_e$,
or that $\vph_e$ is one of the constant functions,
we release the follower by making $\psi_{2x}$ a finite function
different from all functions occurring so far, and stop its enumeration.

We say that $x$ is a {\em follower\/} of $\vph_e$ at stage $s$
if we are trying to make $\psi_x = \vph_e$, i.e.\ $\psi_{x,s} = \vph_{e,s}$.
A follower $x$ of $\vph_e$ is {\em permanent\/} if it is a follower 
of $\vph_e$ at almost every stage. 

For a follower $x$ of $\vph_e$, to {\em release\/} $x$ at stage $s$ means 
that $x$ is no longer a follower of $\vph_e$, and that we define 
$\psi_x$ to be different from all other finite functions $\psi_{y,t}$ that 
have been defined so far, i.e.\ with $y\neq x$, $t\leq s$, and with 
either $0<y<s$ or $y$ odd, 
by making it a finite function incompatible to these. 
We will have that $\psi_0$ is the empty function, and all other 
$\psi_x$ will have nonempty domain. Since we work in $\omega^{<\omega}$ 
it will always be possible to find incompatible strings for finite 
functions with nonempty domain. To make $\psi_x$ incompatible with all 
$\psi_y$ with $y$ odd, it suffices to make it nonconstant.

The formal construction is as follows. 
At stage $s=0$, define $\psi_{2x+1}$ as in \eqref{k} for every~$x$.
Also define $\psi_0 = \emptyset$ to be the empty function. 
In the rest of the construction, we only consider $\vph_e$ with 
nonempty domain. 

At stage $s>0$ of the construction we do the following.

For every $x<s$, if $x$ is a follower of $\vph_e$, 
we release $x$ if $x$ is even and one of the following hold:
\begin{itemize}
 
\item there is $i<e$ with $\vph_{i,s} \restr x = \vph_{e,s} \restr x$. 
(In this case $e$ does not look like a minimal code.)

\item for some follower $y$ already released, $\psi_{y,s} = \psi_{x,s}$. 
(In this case $\vph_e$ might equal the finite function $\psi_y$, 
and we have to avoid the duplication.)

\item $\vph_{e,s} \restr x$ is a constant function, 
i.e.\ $\ex a \, \fa n<x \, (\vph_{e,s}(n) =a)$. 
(Since the constant functions are already covered by the $\psi_{2x+1}$.)

\end{itemize}

If $s = \la e,t\ra$, and $\vph_{e,s}\neq \emptyset$, and $\vph_e$ currently 
does not have a follower, pick the smallest even $x$ that has not yet been used 
as a follower, and appoint $x$ as a follower of $\vph_e$.
Note that this ensures that every nonempty $\vph_e$ has infinitely 
many opportunities of being appointed a follower.

Finally, for every $x$ and $e$ such that $x$ is a follower of $\vph_e$ 
at stage~$s$, define $\psi_{x,s} = \vph_{e,s}$.
This ends the construction.
We verify that the enumeration $\psi_x$ is as desired. 

Claim: $\fa e \, \ex x \, ( \vph_e = \psi_x )$, i.e.\ every unary p.c.\ function 
occurs in the enumeration $\psi_x$. 
To prove the claim, suppose that $e$ is a minimal index of $\vph_e$, and 
that $s_0$ is so large that 
$$
\fa s>s_0 \, \fa x>s_0 \, \fa i< e \, ( \vph_{i,s}\restr x \neq \vph_{e,s}\restr x).
$$
If $\vph_e$  is constant then it is equal to $\psi_{2x+1}$ for some $x$ 
by stage $0$ of the construction. 
Suppose that $\vph_e$ is not constant. 
If $\vph_e$ has a permanent follower $x$ then $\vph_e = \psi_x$. 
Otherwise, $\vph_e$ keeps getting appointed new followers (at stages of 
the form $s = \la e,t\ra$).
Since the first option for releasing a follower $x$ is ruled out after stage $s_0$
by assumption, and the third is ruled out because $\vph_e$ is not constant, 
the only option for releasing $x$ after this stage is the second one, 
namely that $\psi_{y,s} = \psi_{x,s}$ for some already released~$y$. But this 
can happen only once, since all $\psi_{y,s}$ for $y$ released are incompatible. 

Claim: $x\neq y \Longrightarrow \psi_x \neq \psi_y$, i.e.\ $\psi_x$ is an 
enumeration without repetitions.
Namely, the $\psi_x$ for $x$ odd are all different by \eqref{k}.
Note further that every even $x>0$ is eventually used as a follower, since we 
always pick the smallest one not used yet. 

If $x$ is a follower of $\vph_e$, and $\vph_e$ is constant, then 
$x$ will eventually be released by the third reason for release 
in the construction. 
Hence $\psi_x$ is never a constant function for any even~$x$. 
So it suffices to prove the claim for $x$ and $y$ even. 
We check the following cases. 

Suppose that $x$ and $y$ are permanent followers, say $\psi_x = \vph_e$ and 
$\psi_y = \vph_i$. Then $e\neq i$ since $\vph_e$ can have at most one 
permanent follower. W.l.o.g.\ suppose $i<e$. If $\psi_x = \psi_y$ then 
$\vph_e = \vph_i$, hence there is a stage $s$ such that 
$\vph_{i,s} \restr x = \vph_{e,s} \restr x$, causing $x$ to be released, 
contrary to assumption. Hence we must have $\psi_x \neq \psi_y$.

Suppose that both $x$ and $y$ are released at some stage. 
Since all functions $\psi_x$ and $\psi_y$ for different released 
$x$ and $y$ are incompatible, this implies that they are different. 

Finally suppose that one of $x$ and $y$ is permanent and the other is 
released, say $x$ is permanent and $y$ is not.
If $\psi_x = \psi_y$, then $\psi_x$ is a finite function. 
So at some stage $s$ we will have $\psi_{y,s} = \psi_{x,s}$, 
causing $x$ to be released, contradicting the assumption.

This proves the second claim, and the proof of the lemma.
\end{proof}

\begin{theorem} \label{impossible}
\begin{enumerate}[\rm (i)]

\item 
There exists an extensional pas on the set of all partial computable 
unary functions, containing a combinator $k$ as in Theorem~\ref{Feferman}.

\item 
There exists no such a pas with the combinator $s$. 

\end{enumerate}
\end{theorem}
\begin{proof}
(i) 
Using the enumeration $\psi_x$ from Lemma~\ref{lemma}, 
define application as in \eqref{application}. 
This pas is clearly extensional, as the enumeration $\psi_x$ is 1-1.
We show that we have the combinator~$k$.  
Let $k$ be a code such that $\psi_k(a) = 2a+1$ for every~$a$. 
The code $k$ exists because this is a computable function, 
so it occurs in the enumeration. 
Then by \eqref{k} we have 
$$
\psi_{\psi_k(a)}(b) = a, 
$$
hence $k$ is a code of a total function with the property 
$kab = a$ for every~$a$.

For the proof of part (ii),
suppose that $\psi_e$, $e\in\omega$, is a computable 
enumeration without repetitions containing all unary p.c.\ functions, 
and suppose that $\omega$ with the application operator 
\eqref{application} is a pca. 
We make the following observations. 

I. $\{a \in \omega : \psi_a(a)\darrow\}$ is undecidable. 
This is the same as the usual argument for the undecidability of 
the halting problem: 
Suppose that $c$ is a code such that 
$\vph_c(a) \darrow \Leftrightarrow \psi_a(a)\uarrow$ for every~$a$. 
Since the enumeration $\psi_e$ contains a code of every p.c.\ function, 
there exists $e$ such that $\psi_e = \vph_c$. Taking $a=e$ we obtain
a contradiction: 
$\psi_e(e)\darrow \Leftrightarrow 
\vph_c(e)\darrow \Leftrightarrow \psi_e(e)\uarrow$.

II. $\{b \in \omega : \psi_b \text{ is constant zero}\}$ is decidable.\footnote{
Note that for the standard numbering $\vph_e$ of the p.c.\ functions,
the set from II is $\Pi^0_2$-complete.}
Suppose $c$ is a code such that $\psi_c$ is the constant zero function. 
Since codes in the enumeration $\psi_e$ are unique, 
$\psi_b$ is constant zero if and only if $b=c$.

Since we have assumed that $\omega$ with application \eqref{application} 
is a pca, we have combinatory completeness (see Definition~\ref{def:pca}), 
which is an analogue of the S-m-n-theorem. 
Using this we can reduce I to II, and thus we obtain a contradiction. 
Namely, consider the term 
$$
t(x,y) = 0\cdot (x\cdot x). 
$$
Here $0\cdot x$ should be read as the constant zero function applied to $x$ 
(which happens to be the same notation as multiplying with $0$).
By combinatory completeness, there exists $f\in\omega$ such 
that for every $a$ and $c$ in $\A$, 
$fa\darrow$ and $fac \simeq 0(aa) = 0\cdot \psi_a(a)$.
So we have that $\psi_a(a)\darrow$ if and only if 
$\psi_{fa}$ is the constant zero function.
Because $fa = \psi_f(a)$ is a total computable function, this 
constitutes an $m$-reduction from I to II. 
Since the set from II is decidable, it follows that the one 
from I is also decidable, contradicting what we proved above.
\end{proof}

\begin{corollary} \label{cor:imp}
There does not exist an extensional pca on the set of all 
p.c.\ functions (with application the intended one)
\end{corollary}
\begin{proof}
By Theorem~\ref{Feferman}, such a pca would have to contain combinators 
$s$ and $k$, which is impossible by Theorem~\ref{impossible}~(ii).
\end{proof}

As we mentioned above, Kreisel's suggestion was about models of BRFT, 
not pca's. However, since every BRFT gives rise to a pca 
(cf.\ the introduction), 
Corollary~\ref{cor:imp} also precludes the use of Friedberg's result 
to construct extensional models of BRFT.

\section{Inseparability}\label{sec:insep}

In this section we show that every pca has computably 
inseparable subsets $A$ and $B$. This is completely analogous to 
the situation in classical computability theory, even though the 
sets $A$ and $B$ may not always be representable in the pca. 
We use this in the following sections when we discuss 
elements without total extensions. 

\begin{definition} \label{def:separable}
Let $\A$ be a pca. We call a pair of disjoint subsets $A,B\subseteq\A$
{\em computably separable\/} if there exists a decidable subset 
$C\subseteq \A$ such that $A\subseteq C \subseteq \cmp{B}$, 
and {\em computably inseparable\/} otherwise.
\end{definition}

Define
\begin{align*}
A = \{ a \in \A \mid aa\darrow = 0 \}, \\
B = \{ a \in \A \mid aa\darrow = 1 \}.
\end{align*}

\begin{proposition} \label{prop:inseparable}
The sets $A$ and $B$ are computably inseparable in $\A$. 
\end{proposition}
\begin{proof}
Suppose that $A\subseteq C \subseteq \cmp{B}$ and that
$C$ is decidable by $c\in\A$. Then 
\begin{align*}
c\in C \Longrightarrow cc\darrow = 1 \Longrightarrow c\in B \Longrightarrow c\notin C,\\
c\notin C \Longrightarrow cc\darrow = 0 \Longrightarrow c\in A \Longrightarrow c\in C,
\end{align*}
and we have a contradiction. 
\end{proof}

Note that for $\K_1$, the set $A$ and $B$ are the standard example
of a pair of computably inseparable c.e.\ sets. 
We note that the sets $A$ and $B$ need not always be c.e.\ in $\A$. 
A sufficient condition for $A$ and $B$ to be c.e.\ is that every 
singleton $\{a\}$ is c.e.\ in $\A$. 
To see that this implies that $A$ is c.e., suppose that 
$e\in\A$ is such that $ea\darrow \Leftrightarrow aa\darrow$. 
Since $\{0\}$ is c.e., there exists $d\in\A$ such that 
$da\darrow \Leftrightarrow a=0$. Then we have 
$$
d(ea)\darrow \Leftrightarrow ea\darrow = 0 \Leftrightarrow aa\darrow =0,
$$
hence $A$ is c.e.\ in $\A$.
The condition that every singleton is c.e.\ holds in $\K_1$ and $\K_2$.

Scott (cf.\ \cite[Theorem 6.6.2]{Barendregt}) proved that for the set of 
terms $\Lambda$ in the lambda calculus, any pair of disjoint subsets that 
are closed under equality is computably inseparable. 
Note that this refers to ordinary computable inseparability in $\omega$, 
using a suitable coding of lambda-terms \cite[Definition 6.5.6]{Barendregt}.
Note that Definition~\ref{def:separable} is more general, as it also applies 
to uncountable domains.

\section{Elements without total extensions} \label{sec:tot}

\begin{definition}
For elements $b$ and $f$ of a pca $\A$, we say that $f$ is a 
{\em total extension\/} of $b$ if $f$ is total and for every $a\in\A$, 
$$
ba\darrow \; \Longrightarrow \; fa = ba.
$$
\end{definition}

It is well-known that there exist p.c.\ functions without total computable 
extensions. This follows e.g.\ from the existence of computably inseparable 
c.e.\ sets. The existence of inseparable sets from 
Proposition~\ref{prop:inseparable} does not immediately yield the same 
result for pca's, as these sets do not have to be c.e.\ in $\A$. 
To obtain elements without total extensions, an extra property is needed.

\begin{definition} \label{def:sep}
We say that $0$,$1$ are {\em separable\/} in $\A$ if there exists a 
total $0$-$1$-valued $c\in\A$ such that for every $a\in\A$, 
\begin{align*}
ca = 0 &\Longrightarrow a\neq 1 \\
ca = 1 &\Longrightarrow a\neq 0.
\end{align*} 
\end{definition}

Note that separability of $0$,$1$ in $\A$ implies that $0\neq 1$, and 
that it is equivalent to the statement that the subsets $\{0\}$, $\{1\}$ 
are computably separable. 
This provides a {\em constructive\/} way to verify for 
every element $a\in\A$ the formula $a\neq 0 \vee a\neq 1$.

In Definition~\ref{def:sep} we have used $0$ and $1$, i.e.\ the 
numerals $\bar 0$ and $\bar 1$ (cf.\ the discussion 
in section~\ref{sec:comp}), but the notion of separability 
would apply to any other pair of elements from~$\A$.

Separability of $0$ and $1$ is satisfied in $\K_1$ and $\K_2$, 
but not in every pca. For example, it does not hold in the 
$\lambda$ calculus. By Corollary~\ref{cor} below, $0$ and $1$ 
are inseparable in any total pca, and by 
Theorem~\ref{thm:nontotalexample} there also exist nontotal 
examples where this is the case.

\begin{theorem} \label{thm:total}
Suppose that $\A$ is a pca such that $0$,$1$ are separable in $\A$.
Then there exists $b\in\A$ without a total extension $f\in \A$. 
\end{theorem}
\begin{proof}
Define $ba = aa$. (Note that such a $b$ exists by combinatory completeness
applied to the term $t(x) = xx$.)
Suppose that $f\in\A$ is a total extension of $b$, 
and let $c\in\A$ be a total 0-1-valued separation of $0$,$1$ as in 
Definition~\ref{def:sep}.
Then $\hat f a = c(fa)$ is also 0-1-valued, and again $\hat f \in \A$ 
by combinatory completeness. Now 
\begin{align*}
aa\darrow = 0 \Longrightarrow ba\darrow = 0 \Longrightarrow fa=0 \Longrightarrow \hat f a = c(fa)=0,\\
aa\darrow = 1 \Longrightarrow ba\darrow = 1 \Longrightarrow fa=1 \Longrightarrow \hat f a = c(fa)=1,
\end{align*}
and hence $\hat f$ is a total 0-1-valued extension of~$b$. 
But this contradicts the computable inseparability of the sets 
$A$ and $B$ from Proposition~\ref{prop:inseparable}. 
\end{proof}

Note that the proof of Theorem~\ref{thm:total} still does not 
require the sets $A$ and $B$ to be c.e.\ in $\A$.

\begin{corollary} \label{cor}
In any total pca $\A$ (i.e.\ in any combinatorial algebra), 
$0$ and $1$ are inseparable. 
\end{corollary}
\begin{proof}
If $0$,$1$ are separable in $\A$ then by Theorem~\ref{thm:total} 
there exists an element without a total extension, which is 
clearly impossible if $\A$ is total. 
\end{proof}

By Corollary~\ref{cor}, if $0$,$1$ are separable in $\A$ then 
$\A$ is not total. The converse of this does not hold by the 
next theorem. 

\begin{theorem} \label{thm:nontotalexample}
There exists a nontotal pca $\A$ in which $0$,$1$ are inseparable.
\end{theorem}
\begin{proof}
Let $\A$ be any nontrivial 
total pca, and let $f$ be representable in $\A$. 
(We can simply take $f\in\A$.) 
We use the construction of the relativized pca $\A[f]$ from 
section~\ref{sec:relativization}. 
By van Oosten~\cite[Corollary 2.3]{vanOosten2006}, 
the pca $\A[f]$ is  never total, even if $\A$ is total. 
(This is due to the different interpretation of application in $\A[f]$, 
which allows for the definition of undefined functions.)
It is easy to verify that $\A[f]$ has the same total functions as $\A$.
(Simply replace queries to the oracle $f$ by computations in $\A$.)
So if $0$ and $1$ are separable in $\A[f]$ by a total 0-1-valued function $c$, 
then the same must hold in $\A$. But $\A$ is total, hence by 
Corollary~\ref{cor}, $0$,$1$ are inseparable in $\A$.
\end{proof}

Consider the following statements about a pca $\A$:
\begin{enumerate}[\rm (i)]

\item $0$,$1$ are separable in $\A$.

\item The function $ba = aa$ has no total extension in $\A$.

\item There exists an element in $\A$ without total extension in $\A$.

\item $\A$ is not total.

\end{enumerate}

We have (i)$\Rightarrow$(ii)$\Rightarrow$(iii)$\Rightarrow$(iv):
The first implication follows from the proof of Theorem~\ref{thm:total},
and the others are obvious.
In fact, (ii)$\Leftrightarrow$(iii), as can be seen as follows.
The application function $d\la a,b\ra = ab$ is universal, so it suffices to prove that
if $ba=aa$ has a total extension, then so has $d$.
Suppose that $f$ is a total extension of $b$. Then
$$
ab\darrow \Longleftrightarrow g\la a,b\ra ( g\la a,b\ra )\darrow = f(g\la a,b\ra)
$$
so $f(g\la a,b\ra)$ is a total extension of~$d$.

By Theorem~\ref{thm:nontotalexample} we have that (iv)$\not\Rightarrow$(i), 
but we can in fact say more.
In section~\ref{sec:Kreisel} we discussed Kreisel's suggestion for 
constructing a nontotal extensional pca from a Friedberg numbering. 
Despite the failure of this (Corollary~\ref{cor:imp}), 
such pca's $\A$ do exist, as was proven in 
Bethke and Klop~\cite{BethkeKlop}.
Since $\A$ is extensional, every element in $\A$ has a total extension in $\A$, 
as was proven in \cite{BarendregtTerwijn2}.\footnote{
It follows from Proposition~5.2 in \cite{BarendregtTerwijn2} that 
if $\A$ is extensional then the identity on $\A$ is precomplete, 
which is equivalent to the statement that  
every element in $\A$ has a total extension in $\A$.}
Since $\A$ is nontotal, we have (iv)$\not\Rightarrow$(iii).
At the moment we do not know whether (ii)$\not\Rightarrow$(i).

The negation of item (iii) does not imply that $\A$ has a total completion in 
the sense of Bethke et al.~\cite{BethkeKlopdeVrijer}, as one might think. 
Indeed, $\neg$(iii) implies that in particular the application 
function $d\la a,b\ra = ab$ has a total extension $h\in \A$, 
but this total extension $h$ does not have to respect the 
structure of the combinator~$s$. 
In fact, if we let $\A$ be nontotal and extensional as above, 
by extensionality $\neg$(iii) holds in $\A$ 
(cf.\ \cite[Proposition 5.2]{BarendregtTerwijn2}), 
but $\A$ is not completable by \cite[Theorem 7.2]{BethkeKlopdeVrijer}.

\section{Precompleteness and 1-1 numberings}\label{sec:numberings}

In this section we consider numberings without repetitions, 
often simply called {\em 1-1 numberings}.

A numbering of the p.c.\ functions that is equivalent to 
the standard numbering is called {\em acceptable\/} \cite[p215]{Odifreddi}.
Rogers~\cite{Rogers1967} showed that acceptable numberings are precisely 
those for which the enumeration theorem and parametrization 
(= the S-m-n-theorem) hold.
It also follows from this that for any acceptable numbering the 
padding lemma holds, ensuring that every p.c.\ function has infinitely 
many codes. 
In particular, we see that no 1-1 numbering of the p.c.\ 
functions (such as Friedberg's numbering) is acceptable. 
For more on 1-1-numberings see Kummer~\cite{Kummer}.

A general theory of countable numberings was initiated by 
Ershov \cite{Ershov}. 
A numbering of a set $S$ is simply a surjective function 
$\gamma:\omega\rightarrow S$. 
In particular, Ershov  introduced the notion of a {\em precomplete\/}
numbering on~$\omega$,
and he proved in \cite{Ershov2} that Kleene's recursion theorem 
holds for every precomplete numbering.
Barendregt and Terwijn \cite{BarendregtTerwijn} extended the 
setting to partial combinatory algebra by defining the notion 
of a {\em generalized numbering\/} as a surjective function 
$\gamma:\A\rightarrow S$, where $\A$ is a pca and $S$ is a set. 
The notion of precompleteness for generalized numberings was 
also defined in \cite{BarendregtTerwijn}. It is equivalent to the 
following definition:

\begin{definition} \label{def:precomplete}
A generalized numbering $\gamma \colon \A \rightarrow S$ is
{\em precomplete\/} if for every $b{\in} \A$
there exists a total element $f{\in} \A$ such that
for all $a{\in} \A$,
\begin{equation} \label{precomplete2}
b{a}\darrow \; \Longrightarrow \; f{a} \sim_\gamma b{a}. 
\end{equation}
In this case, we say that {\em $f$ totalizes $b$ modulo~$\sim_\gamma$\/}.
\end{definition}

Ershov's notion of precomplete numbering is obtained from this 
by taking for $\A$ Kleene's first model $\K_1$.
Section 5 of \cite{BarendregtTerwijn2} studies the relations between 
combinatory completeness,
extensionality, and precompleteness of generalized numberings.

The standard numbering of the p.c.\ functions is precomplete by the 
S-m-n-theorem, and since every acceptable numbering is equivalent to 
the standard numbering it follows that 
acceptable numberings are precomplete. 
On the other hand, Friedberg's 1-1 numbering is not precomplete. 
We generalize this fact in Theorem~\ref{precomplete1-1} below.

Precompleteness is connected to the question which elements have total 
extensions, studied in section~\ref{sec:tot}.
For example, the identity $\gamma_\A : \A\rightarrow\A$ is precomplete
if and only if every element $b\in\A$ has a total extension $f\in\A$.

\begin{theorem} \label{precomplete1-1}
Suppose $\gamma:\A\rightarrow S$ is a precomplete generalized numbering, 
and that $0$,$1$ are separable in $\A$. Then $\gamma$ is not 1-1. 
\end{theorem}
\begin{proof}
Suppose that $\gamma$ is precomplete and 1-1, 
and suppose that $b\in\A$. 
Since $\gamma$ is precomplete, there exists $f\in \A$ that totalizes 
$b$ modulo $\sim_\gamma$. As $\gamma$ is 1-1 we have 
$$
ba\darrow \, \Longrightarrow fa \sim_\gamma ba  \Longrightarrow fa = ba
$$
for every~$a$. Hence every $b\in \A$ has a total extension $f\in\A$. 
But this contradicts Theorem~\ref{thm:total}.
\end{proof}

\section{Arslanov's completeness criterion}\label{sec:Arslanov}

Ershov~\cite{Ershov2} showed that Kleene's recursion theorem holds 
for any precomplete numbering $\gamma:\omega\rightarrow S$. 
Working in another direction, Feferman~\cite{Feferman} proved that the 
recursion theorem holds in any pca $\A$. 
In Barendregt and Terwijn~\cite{BarendregtTerwijn}, the fixed point theorems 
of Ershov and Feferman were combined by proving a fixed point theorem for 
precomplete {\em generalized\/} numberings $\gamma:\A\rightarrow S$, 
that instead of $\omega$ have an arbitrary pca $\A$ as a basis. 
The following diagram summarizes the various possible settings of 
the recursion theorem. 
$$
\begin{array}{ccc}
\makebox[0.2cm][r]{pca $\A$} 
& \longrightarrow & 
\makebox[2.2cm][l]{$\gamma:\A\rightarrow S$ generalized numbering}\\
& &  \\
\uparrow  &   & \uparrow \\
& & \\
\omega & \longrightarrow & 
\makebox[2.1cm][l]{$\gamma:\omega\rightarrow S$ numbering}\\
\end{array}
$$

Now another famous extension of the recursion theorem is 
Arslanov's completeness criterion \cite{Arslanov}, which extends 
the recursion theorem from computable functions to the class of all 
functions that are computable from a Turing-incomplete c.e.\ set. 
Explicitly, suppose that $A\subseteq \omega$ is a c.e.\ set such 
that $K\not\leq_T A$, and suppose that $f$ is an $A$-computable 
function. Then there exists $e\in\omega$ such that for all $x\in\omega$,
$$
\vph_{f(e)}(x) \simeq \vph_e(x).
$$
In Barendregt and Terwijn~\cite{BarendregtTerwijn} it was shown that 
Arslanov's completeness criterion also holds for any precomplete 
numbering. (In contrast to this, it is open whether the joint generalization 
from \cite{Terwijn} also holds for every precomplete numbering.) 
This prompts the question whether Arslanov's completeness criterion 
also holds for generalized numberings. A first step would be to prove 
an analog of Arslanov's result for pca's. Using the concepts of 
section~\ref{sec:relativization}, we can formulate such an analog as 
follows. 

Let $\A$ be a pca, and suppose that $A$ is c.e.\ in $\A$ such that 
$K\notin \A[A]$, where $K$ is the halting set in $\A$ defined 
in section~\ref{sec:relativization}.  
Note that this is the analog of of stating that $A$ is a c.e.\ set 
that is not Turing complete. 
Now Arslanov's result says that any $A$-computable function $f$ has a 
fixed point, which translates to the following. 
Suppose that $f\in \A[A]$ is total. Then there exists $e\in \A$ such 
that for all $x\in \A$, 
$$
f \cdot_A e \cdot x \simeq e\cdot x.
$$
Here $\cdot_A$ denotes application in $\A[A]$ and 
$\cdot$ denotes application in $\A$.

\begin{question}
Does this analog of Arslanov's completeness criterion hold for 
every pca?
\end{question}

\end{document}